\documentclass[12pt, reqno]{amsart}
\usepackage{amsmath, amsthm, amscd, amsfonts, amssymb, graphicx, color}
\usepackage[bookmarksnumbered, colorlinks, plainpages]{hyperref}
\hypersetup{colorlinks=true,linkcolor=red, anchorcolor=green,
citecolor=cyan, urlcolor=red, filecolor=magenta, pdftoolbar=true}

\newtheorem{theorem}{Theorem}[section]

\newtheorem{proposition}[theorem]{Proposition}
\newtheorem{corollary}[theorem]{Corollary}
\theoremstyle{definition}
\newtheorem{definition}[theorem]{Definition}

\theoremstyle{remark}
\newtheorem{remark}[theorem]{Remark}
\numberwithin{equation}{section}

\begin{document}

\setcounter{page}{1}

\title[Numerical radius orthogonality in $C^*$-algebras]
{Numerical radius orthogonality in $C^*$-algebras}

\author[A. Zamani and Pawe\l{} W\'ojcik]
{Ali Zamani$^{1,*}$ and Pawe\l{} W\'ojcik$^2$}

\address{$^*$ Corresponding author, $^1$Department of Mathematics, Farhangian University, Tehran, Iran}
\email{zamani.ali85@yahoo.com}

\address{$^2$ Institute of Mathematics, Pedagogical University of Cracow, Podchor\c a\.zych~2, 30-084 Krak\'ow, Poland}
\email{pawel.wojcik@up.krakow.pl}


\begin{abstract}
In this paper we characterize the Birkhoff--James orthogonality with respect
to the numerical radius norm $v(\cdot)$ in $C^*$-algebras.
More precisely, for two elements $a, b$ in a $C^*$-algebra $\mathfrak{A}$,
we show that $a\perp_{B}^{v} b$ if and only if for each $\theta \in [0, 2\pi)$,
there exists a state $\varphi_{_{\theta}}$ on $\mathfrak{A}$ such that
$|\varphi_{_{\theta}}(a)| = v(a)$ and $\mbox{Re}\big(e^{i\theta}\overline{\varphi_{_{\theta}}(a)}\varphi_{_{\theta}}(b)\big)\geq 0$.
Moreover, we compute the numerical radius derivatives in $\mathfrak{A}$.
In addition, we characterize when the numerical radius norm of the sum of two (or three)
elements in $\mathfrak{A}$ equals the sum of their numerical radius norms.
\end{abstract} \maketitle
\section{Introduction and Preliminaries}\label{s1}
Throughout this paper, let $\mathfrak{A}$ be a unital $C^*$-algebra with unit denoted by $e$,
and $\mathbb{B}(\mathcal{H})$ be the $C^*$-algebra of all bounded linear operators on a complex
Hilbert space $\big(\mathcal{H}, \langle \cdot, \cdot\rangle\big)$.
We denote by $\mathfrak{A}'$ the dual space of $\mathfrak{A}$.
A linear functional $\varphi \in \mathfrak{A}'$ is said to be {\it positive},
and write $\varphi \geq 0$, if $\varphi(a^*a) \geq 0$ for all $a\in \mathfrak{A}$.
By $\mathcal{S}(\mathfrak{A})$ we denote the set of all normalized states of $\mathfrak{A}$, that is,
\begin{align*}
\mathcal{S}(\mathfrak{A}) = \big\{\varphi \in \mathfrak{A}':\,
\varphi \geq 0 \quad \mbox{and} \quad \varphi (e) = \|\varphi\| = 1\big\}.
\end{align*}
The {\it numerical range of an element} $a\in \mathfrak{A}$ is
$V(a) = \big\{\varphi(a): \, \varphi \in \mathcal{S}(\mathfrak{A})\big\}$.
It is a nonempty compact and convex set of the complex plane $\mathbb{C}$,
and its maximum modulus is the {\it numerical radius} $v(a)$ of $a$, that is,
\begin{align*}
v(a) = \sup\big\{|\xi|: \, \xi \in V(a)\big\}.
\end{align*}
This definition generalizes the classical numerical radius in the sense
that the numerical radius $v(A)$ of a Hilbert space operator $A$
(considered as an element of a $C^*$-algebra $\mathbb{B}(\mathcal{H})$)
coincides with classical numerical radius
\begin{align*}
w(A) = \sup\big\{|\langle Ax, x\rangle|: \, x\in \mathcal{H},\|x\| = 1\big\}.
\end{align*}
It is well known that $v(\cdot)$ defines a norm on $\mathfrak{A}$, which is equivalent
to the $C^*$-norm $\|\cdot\|$.
In fact, the following inequalities hold for every $a\in \mathfrak{A}$:
\begin{align}\label{I.1.1}
\frac{1}{2}\|a\| \leq v(a)\leq \|a\|.
\end{align}
For more material about the numerical radius, we refer the reader to \cite{B.D} and \cite{G.R}.

The usual way to define the orthogonality in $\mathfrak{A}$
is by means of the $C^*$-valued inner product: for elements $a, b$ of $\mathfrak{A}$ we say that $a$ is
orthogonal to $b$, and we write $a\perp b$, if $a^*b = 0$.
Another concept of orthogonality in $\mathfrak{A}$ is the Birkhoff--James orthogonality (see \cite{Bi,J.3}).
Recall that, an element $a\in \mathfrak{A}$ is said to be {\it Birkhoff--James orthogonal} to another element
$b\in \mathfrak{A}$, in short
$a\perp_{B} b$, if $\|a+\lambda b\|\geq\|a\|$ for all $\lambda\in\mathbb{C}$.

As a natural generalization of the notion of
Birkhoff--James orthogonality in $C^*$-algebras,
the concept of strong Birkhoff–-James orthogonality
was introduced in \cite{A.R.1}.
When $a$ and $b$ are elements of $\mathfrak{A}$, $a$ is orthogonal to $b$ in
the strong Birkhoff--James sense, in short $a\perp_{B}^{s} b$,
if $\|a + bc\|\geq\|a\|$ for all $c\in\mathfrak{A}$.

The characterizations of the (strong) Birkhoff–-James orthogonality for elements of a
$C^*$-algebra by means of the states are known.
For elements $a, b$ of $\mathfrak{A}$ the following results were obtained in \cite{A.R.1,A.R.2,B.G}:
\begin{align*}
a\perp_{B} b \Leftrightarrow \big(\exists \,\varphi \in \mathcal{S}(\mathfrak{A}):
\varphi(a^*a)=\|a\|^2 \,\,\mbox{and}\,\, \varphi(a^*b)=0\big)
\end{align*}
and
\begin{align*}
a\perp_{B}^{s} b \Leftrightarrow \big(\exists \,\varphi \in \mathcal{S}(\mathfrak{A}):
\varphi(a^*a)=\|a\|^2 \,\,\mbox{and}\,\, \varphi(a^*bb^*a)=0\big).
\end{align*}

In the next section, inspired by the numerical radius parallelism
in \cite{Z.1}, we discuss the Birkhoff--James orthogonality with
respect to the numerical radius norm in $\mathfrak{A}$.
We show that this relation can be characterized
in terms of states acting on $\mathfrak{A}$ (Theorem \ref{T.7.2}).
Some other related results are also discussed.
Particularly, we prove that $v(a + b) = v(a) + v(b)$ if and only if
there exists a state $\varphi$ on $\mathfrak{A}$ such that
$\overline{\varphi(a)} \varphi(b) = v(a)v(b)$.
In addition, we compute the numerical radius derivatives
in $\mathfrak{A}$ (Theorem \ref{T.norm.der}).
\section{Main Results}\label{s2}
We start our work with the following definition.
\begin{definition}\label{D.1.2}
An element $a\in\mathfrak{A}$ is called
the {\it numerical radius Birkhoff–-James orthogonal} to another
element $b \in\mathfrak{A}$, denoted by $a\perp_{B}^{v} b$, if
$v(a + \lambda b) \geq v(a)$ for all $\lambda\in\mathbb{C}$.
\end{definition}
Notice that the relations $\perp_{B}$ and $\perp_{B}^{v}$ are not comparable, in general.
As an example, one can take the
$C^*$-algebra $\mathfrak{A}$ of all complex $2\times 2$ matrices and let
$a = \begin{bmatrix}
i & 1 \\
0 & i
\end{bmatrix}$,
$b = \begin{bmatrix}
-2i & 0 \\
1 + \sqrt{5} & 0
\end{bmatrix}$,
$c = \begin{bmatrix}
0 & 0 \\
i & 0
\end{bmatrix}$,
and
$d = \begin{bmatrix}
1 & 0 \\
1 & 2
\end{bmatrix}$.
Then simple computations show that $a\perp_{B}b$ but $a \not\perp_{B}^{v}b$
and also, $c\perp_{B}^{v}d$ but $c \not\perp_{B}d$.

Note that these relations are coincident for certain elements in $C^*$-algebras.
For example, if $a\in\mathfrak{A}$ is normal, then $v(a) = \|a\|$ (see \cite[p. 44]{B.D})
and hence the condition $a\perp_{B}^{v} b$
implies $a\perp_{B} b$ for all $b\in\mathfrak{A}$. Indeed, for every $\lambda\in \mathbb{C}$,
by (\ref{I.1.1}), we have
\begin{align*}
\|a + \lambda b\| \geq v(a + \lambda b) \geq v(a) = \|a\|.
\end{align*}
Furthermore, if $a^2 = 0$, then by \cite[Corollary 2.5]{Z.1} $v(a) = \frac{1}{2}\|a\|$
and so the condition $a\perp_{B}b$
implies $a\perp_{B}^{v} b$ for all $b\in\mathfrak{A}$. Indeed, for every $\lambda\in \mathbb{C}$,
again by (\ref{I.1.1}) it follows that
\begin{align*}
v(a + \lambda b) \geq \frac{1}{2}\|a + \lambda b\| \geq \frac{1}{2}\| a\| = v(a).
\end{align*}
\begin{remark}
Let $a\in \mathfrak{A}$.
Define $f: \mathfrak{A}'\longrightarrow \mathbb{C}$ by the formula $f(\varphi) = \varphi(a)$.
Then the function $f$ is weak*-continuous. Therefore the function $g: \mathfrak{A}'\longrightarrow \mathbb{C}$,
given by $g(\varphi): = |f(\varphi)| = |\varphi(a)|$, is also weak*-continuous.
Moreover, the set of normalized states $\mathcal{S}(\mathfrak{A})$ is weak*-compact.
Since the function $g$ is weak*-continuous, it follows that the function
$g|_{_{\mathcal{S}(\mathfrak{A})}}: \mathcal{S}(\mathfrak{A})\longrightarrow \mathbb{C}$
attains its maximum. Therefore, for a given element $a\in \mathfrak{A}$ there is $\varphi \in\mathcal{S}(\mathfrak{A})$
such that $|\varphi(a)| = v(a)$.
\end{remark}
The following proposition states some basic properties of the relation $\perp_{B}^{v}$.
\begin{proposition}\label{P.2.2}
Let $a, b\in \mathfrak{A}$.
Then the following statements are equivalent:
\begin{itemize}
\item[(i)] $a\perp_{B}^{v} b$.
\item[(ii)] $a^*\perp_{B}^{v} b^*$.
\item[(iii)] $\alpha a\perp_{B}^{v} \beta b$ for all $\alpha, \beta \in \mathbb{C}$.
\item[(iv)] $ac\perp_{B}^{v} bc$ for every unitary element $c$ in the center of $\mathfrak{A}$.
\end{itemize}
If $a, b$ are self-adjoint, then each one of these assertions is also equivalent to
\begin{itemize}
\item[(v)] $v(a + rb)\geq v(a)$ for all $r \in \mathbb{R}$.
\end{itemize}
\end{proposition}
\begin{proof}
It is a basic fact that the norm $v(\cdot)$ is self-adjoint (i.e., $v(c^*) = v(c)$ for
every $c\in \mathfrak{A}$) and so the equivalence (i)$\Leftrightarrow $(ii) is trivial.
The equivalence (i)$\Leftrightarrow$(iii)
immediately follows from the definition of the relation $\perp_{B}^{v}$.
The implication (iv)$\Rightarrow $(i) is also trivial.
It is therefore enough to prove the implication (i)$\Rightarrow $(iv).

Suppose that (i) holds. Let $c$ be a unitary element in the center of $\mathfrak{A}$.
By the first part of the proof of \cite[Theorem 3.4]{Z.1},
it follows that $v(dc) = v(d)$ for all $d \in \mathfrak{A}$.
So, we conclude that
\begin{align*}
v(ac + \lambda bc) = v\big((a + \lambda b)c\big) = v(a + \lambda b)\geq v(a) = v(ac),
\end{align*}
for all $\lambda\in \mathbb{C}$. Thus $ac\perp_{B}^{v} bc$.

Now, let $a, b$ be self-adjoint. Suppose (v) holds. Let $\lambda = t + is \in \mathbb{C}$ and let $\psi$ be
a state on $\mathfrak{A}$ such that $|\psi(a + tb)| = v(a + tb)$.
We have
\begin{align*}
v^2(a + \lambda b) &\geq |\psi(a + \lambda b)|^2
= |\psi(a + tb) + i\psi(sb)\big)|^2
\\&= |\psi(a + tb)|^2 + |\psi(sb)|^2
\\&\geq |\psi(a + tb)|^2
= v^2(a + tb) \geq v^2(a),
\end{align*}
and so $v(a + \lambda b)\geq v(a)$. Thus $a\perp_{B}^{v} b$.
The converse, that is, (i) implies (v), is obvious.
\end{proof}
In the following result we characterize a positive-real version of the numerical
radius Birkhoff–-James orthogonality.
Our approach is similar to the one given in \cite{M.P.S}.
\begin{theorem}\label{T.3.2}
Let $a, b\in \mathfrak{A}$.
Then the following statements are equivalent:
\begin{itemize}
\item[(i)] $v(a + rb)\geq v(a)$ for all $r \in \mathbb{R}^{+}$.
\item[(ii)] There exists a state $\varphi$ on $\mathfrak{A}$ such that
\begin{align*}
|\varphi(a)| = v(a) \quad \mbox{and} \quad {\rm Re}(\overline{\varphi(a)}\varphi(b)\big)\geq 0.
\end{align*}
\end{itemize}
\end{theorem}
\begin{proof}
(i)$\Rightarrow$(ii) Let $v(a + rb)\geq v(a)$ for all $r \in \mathbb{R}^{+}$.
We may assume that $v(a) \neq 0$ otherwise (ii) trivially holds. Thus
there is $\varepsilon_o\!\in\!(0,1)$ such
that $v(a) - \varepsilon^2 \geq 0$ for all $\varepsilon\!\in\!(0,\varepsilon_o)$.
So, it follows that
\begin{align}\label{I.1.T.3.2}
v(a + \varepsilon b) \geq v(a) \geq v(a) - \varepsilon^2 \geq 0
\end{align}
for all $\varepsilon\!\in\!(0,\varepsilon_o)$.
On the other hand, there exists a state $\varphi_{\varepsilon}$
on $\mathfrak{A}$ such that $|\varphi_{\varepsilon}(a + \varepsilon b)| = v(a + \varepsilon b)$.
So, by (\ref{I.1.T.3.2}) it follows that
\begin{align*}
v(a) + \varepsilon v(b) \geq |\varphi_{\varepsilon}(a)| + \varepsilon |\varphi_{\varepsilon}(b)|
\geq |\varphi_{\varepsilon}(a + \varepsilon b)| = v(a + \varepsilon b) \geq v(a).
\end{align*}
Since the set $\mathcal{S}(\mathfrak{A})$ is weak*-compact, we may
assume that $\varphi_{\varepsilon}\stackrel{w^*}{\longrightarrow} \varphi_o$ for
some $\varphi_o\!\in\!\mathcal{S}(\mathfrak{A})$, where $\varepsilon \rightarrow 0^{+}$.
Now, letting $\varepsilon \rightarrow 0^{+}$, we get $|\varphi_o(a)| = v(a)$.

Furthermore, from (\ref{I.1.T.3.2}) it follows that
\begin{align*}
v^2(a) + 2\varepsilon\mbox{Re}(\overline{\varphi_{\varepsilon}(a)}\varphi_{\varepsilon}(b)\big) + \varepsilon^2 v^2(b)
&\geq |\varphi_{\varepsilon}(a)|^2 + 2\varepsilon{\rm Re}(\overline{\varphi_{\varepsilon}(a)}\varphi_{\varepsilon}(b)\big)
+ \varepsilon^2 |\varphi_{\varepsilon}(b)|^2
\\&=|\varphi_{\varepsilon}(a + \varepsilon b)|^2= v^2(a + \varepsilon b)
\\& \geq v^2(a) - 2\varepsilon^2v(a) + \varepsilon^4,
\end{align*}
and hence
\begin{align*}
{\rm Re}(\overline{\varphi_{\varepsilon}(a)}\varphi_{\varepsilon}(b)\big)
\geq \frac{\varepsilon^3}{2} - \varepsilon v(a) -\frac{\varepsilon}{2} v^2(b).
\end{align*}
Thus, by letting $\varepsilon \rightarrow 0^{+}$, we obtain
${\rm Re}(\overline{\varphi_o(a)}\varphi_o(b)\big)\geq 0$.

(ii)$\Rightarrow$(i) Suppose (ii) holds. Therefore, for every $r \in \mathbb{R}^{+}$, we have
\begin{align*}
v^2(a + rb) \geq |\varphi(a + rb)|^2
= |\varphi(a)|^2 + 2r {\rm Re}(\overline{\varphi(a)}\varphi(b)\big)+ r^2|\varphi(b)|^2
\geq v^2(a),
\end{align*}
and so $v(a + rb)\geq v(a)$.
\end{proof}
In what follows, we get a very tractable characterization of the numerical
radius Birkhoff–-James orthogonality in the positive cones of $C^*$-algebras.
Recall that the positive elements of $\mathfrak{A}$ are the elements of the form $a^*a$, where $a\in\mathfrak{A}$.
\begin{corollary}\label{C.5.2}
Let $a, b$ be positive elements of $\mathfrak{A}$.
Then the following statements are equivalent:
\begin{itemize}
\item[(i)] $a\perp_{B}^{v} b$.
\item[(ii)] There exists a state $\varphi$ on $\mathfrak{A}$ such that
$\varphi(a) = v(a)$ and $\varphi(b) = 0$.
\end{itemize}
\end{corollary}
\begin{proof}
(i)$\Rightarrow$(ii) Let $a\perp_{B}^{v} b$. By Proposition \ref{P.2.2}, we have $a\perp_{B}^{v} (-b)$.
So, by Theorem \ref{T.3.2} there exists a state $\varphi$ on $\mathfrak{A}$ such that
$|\varphi(a)| = v(a)$ and $\mbox{Re}\big(\overline{\varphi(a)}\varphi(-b)\big) \geq 0$.
Since $a, b$ are positive, we reach that $\varphi(a) = v(a)$ and
\begin{align*}
0 \leq \varphi(b) = \frac{-\mbox{Re}\big(\overline{\varphi(a)}\varphi(-b)\big)}{\varphi(a)}\leq 0.
\end{align*}
Thus $\varphi(b) = 0$.

(ii)$\Rightarrow$(i) Let $\varphi$ be a state on $\mathfrak{A}$ such that
$\varphi(a) = v(a)$ and $\varphi(b) = 0$. Then for every $\lambda \in \mathbb{C}$,
we have
\begin{align*}
v(a + \lambda b)\geq |\varphi(a + \lambda b)| = |\varphi(a) + \lambda \varphi(b)| = \varphi(a) = v(a),
\end{align*}
and hence $a\perp_{B}^{v} b$.
\end{proof}
\begin{remark}\label{R.6.2}
For positive elements $a, b$ of a unital $C^*$-algebra $\mathfrak{A}$,
Komure et al. in \cite[Lemma 2.3]{K.S.T} proved that $a\perp_{B} b$ if and only if
there exists a state $\varphi$ on $\mathfrak{A}$ such that
$\varphi(a) = \|a\|$ and $\varphi(b) = 0$. Therefore, by Corollary \ref{C.5.2}, we conclude that
the relations $\perp_{B} $ and $\perp_{B}^{v} $ are coincident in the positive cones of $C^*$-algebras.
\end{remark}
In a normed linear space $(\mathcal{X},\|\!\cdot\!\|)$,
the {\it Gateaux derivatives of the norm} are given for
$x, y \in\mathcal{X}$ by the two expressions
\begin{center}
$\rho^{\|\cdot\|}_{\pm}(x,y):=\lim\limits_{t\to
0^{\pm}}\frac{\|x+ty\|^2-\|x\|^2}{2t}=\|x\|\!\lim\limits_{t\to
0^{\pm}}\frac{\|x+ty\|-\|x\|}{t}$.
\end{center}
If it will not cause a confusion, we will write $\rho_{\pm}$ instead of $\rho^{\|\cdot\|}_{\pm}$.
When the norm on $\mathcal{X}$ comes from an inner product
$\langle\cdot|\cdot\rangle\colon \mathcal{X}\!\times\! \mathcal{X}\!\to\!\mathbb{R}$, we obtain
$\rho_+(x,y)=\langle x|y\rangle=\rho_-(x,y)$, i.e., functionals
$\rho_+$, $\rho_-$ are nice generalizations of inner
products. By convexity of the norm the above definitions
are meaningful. The mappings $\rho_+$ and $\rho_-$ are called
the {\it norm derivatives} and their following properties, which will be
useful in the present note, can be found, e.g., in \cite{Dragomir}:

(ND1)\quad $\forall x,y\in \mathcal{X}
\quad -\|x\|\,\|y\|\leq \rho_-(x,y)\leq\rho_+(x,y) \leq \|x\|\,\|y\|$;

(ND2)\quad $\forall x,y\in \mathcal{X}\ \forall \alpha\geqslant0
\quad \rho_\pm(\alpha x, y)=\alpha\rho_\pm(x,y)=\rho_\pm(x,\alpha y)$;

(ND3)\quad $\forall x,y\in \mathcal{X}\ \forall \alpha<0
\quad \rho_\pm(\alpha x,y)=\alpha\rho_\mp(x,y)=\rho_\pm(x,\alpha y)$;

(ND4)\quad $\forall x,y\in \mathcal{X}\ \forall \alpha\in\mathbb{R}
\quad \rho_\pm(x,\alpha x+y)=\alpha\|x\|^2+\rho_\pm(x,y)$.

In a real normed space $\mathcal{X}$, we have for arbitrary $x,y\in \mathcal{X}$:

(ND5)\quad $x \perp_{B} y \ \Leftrightarrow \ \rho_-(x,y)\leq 0 \leq \rho_+(x,y)$;

(ND6)\quad $\rho_-(x,y) = 0\ \Rightarrow\ x\perp_{B}y,\quad \rho_+(x,y) = 0\ \Rightarrow\ x\perp_{B}y$.

Moreover, mappings $\rho_+, \rho_-$ are continuous with respect
to the second variable, but not necessarily with respect to the first one.

The condition (ND5) shows that the Birkhoff--James orthogonality is connected with the norm derivatives.
Therefore, in view of Theorem \ref{T.3.2}, it seems to be quite
natural to compute the numerical radius derivatives, i.e. the norm
derivatives in $\mathfrak{A}$ equipped with the norm $v(\cdot)$.
\begin{theorem}\label{T.norm.der}
Let $a, b\in \mathfrak{A}\setminus\{0\}$.
Then the following statements are true:
\begin{itemize}
\item[(i)] $\rho^{v(\cdot)}_+(a,b)=\max\big\{{\rm Re}\big(\overline{\varphi(a)}\varphi(b)\big):
\varphi\!\in\!\mathcal{S}(\mathfrak{A}),\ |\varphi(a)|\!=\!v(a)\big\}$.
\item[(ii)] $\rho^{v(\cdot)}_-(a,b)=\min\big\{{\rm Re}(\overline{\varphi(a)}\varphi(b)\big):
\varphi\!\in\!\mathcal{S}(\mathfrak{A}),\ |\varphi(a)|\!=\!v(a)\big\}$.
\end{itemize}
\end{theorem}
\begin{proof}
Since the proofs are similar we calculate only $\rho^{v(\cdot)}_+(a,b)$.
It follows from \cite[Theorem 15, p.36]{Dragomir} that
\begin{equation}\label{rho-ab}
\rho^{v(\cdot)}_+(a,b)=v(a)\sup\big\{{\rm Re}(\varphi(b)):
\varphi\!\in\!\mathfrak{A}',\ \|\varphi\|\!=\!1,\ \varphi(a)\!=\!v(a)\big\}.
\end{equation}
Fix $\varphi\!\in\!\mathcal{S}(\mathfrak{A})$ such that $|\varphi(a)|\!=\!v(a)$.
Next we define a linear mapping $\psi\colon\mathfrak{A}\!\to\!\mathbb{C}$
by the formula $\psi(\cdot):=\frac{1}{v(a)}\overline{\varphi(a)}\varphi(\cdot)$.
A moments reflection shows that
\begin{equation}\label{incusion-re-phi}
\psi\!\in\!\mathfrak{A}',\quad \|\psi\|\!=\! 1\quad {\rm and}\quad \psi(a)=v(a).
\end{equation}
Combining (\ref{rho-ab}) and (\ref{incusion-re-phi}), we immediately get
\begin{equation}\label{rho-ab-inequality}
\rho^{v(\cdot)}_+(a,b)\geq \sup\big\{{\rm Re}\big(\overline{\varphi(a)}\varphi(b)\big):
\varphi\!\in\!\mathcal{S}(\mathfrak{A}),\ |\varphi(a)|\!=\!v(a)\big\}.
\end{equation}
Now we are going to prove the converse inequality. It follows
from the property (ND4) that
$\rho^{v(\cdot)}_+\left(a, \frac{-\rho^{v(\cdot)}_+(a,b)}{v^2(a)}a+b\right) = 0$.
Applying (ND6) we get
\begin{center}
$v\left(a+ r\left(\frac{-\rho^{v(\cdot)}_+(a,b)}{v^2(a)}a+b\right)\right)\geq v(a),$
\end{center}
for all $r\in \mathbb{R}$. Now by Theorem \ref{T.3.2} there
is a state $\varphi_o\!\in\!\mathcal{S}(\mathfrak{A})$ such
that $|\varphi_o(a)|\! =\! v(a)$  and
\begin{center}
${\rm Re}\left(\overline{\varphi_o(a)}\varphi_o\left(\frac{-\rho^{v(\cdot)}_+(a,b)}{v^2(a)}a+b\right)\right)\geq 0$.
\end{center}
This implies ${\rm Re}\left(\overline{\varphi_o(a)}\varphi_o(a)\frac{-\rho^{v(\cdot)}_+(a,b)}{v^2(a)}\right)+{\rm Re}\left(\overline{\varphi_o(a)}\varphi_o(b)\right)\geq 0$.
Since $\overline{\varphi_o(a)}\varphi_o(a)=v^2(a)$, we
obtain $\rho^{v(\cdot)}_+(a,b)\leq {\rm Re}\big(\overline{\varphi_o(a)}\varphi_o(b)\big)$.
Further, from this inequality  and from (\ref{rho-ab-inequality}) we have
\begin{equation*}
\rho^{v(\cdot)}_+(a,b)=\sup\big\{{\rm Re}\big(\overline{\varphi(a)}\varphi(b)\big):
\varphi\!\in\!\mathcal{S}(\mathfrak{A}),\ |\varphi(a)|\!=\!v(a)\big\}.
\end{equation*}
Finally, since $\rho^{v(\cdot)}_+(a,b)={\rm Re}\big(\overline{\varphi_o(a)}\varphi_o(b)\big)$, the
word ``sup" can be replaced by the word ``max". The proof is complete.
\end{proof}
For $A, B\in \mathbb{B}(\mathcal{H})$, Bhatia and \v{S}emrl in \cite[Remark 3.1]{B.S}
and Paul in \cite[Lemma 2]{Pa} independently proved that $A\perp_{B} B$ if and only if
there exists a sequence $\{x_n\}$ of unit
vectors in $\mathcal{H}$ such that
\begin{align*}
\lim_{n\rightarrow\infty} \|Ax_n\| = \|A\| \quad
\mbox{and} \quad \lim_{n\rightarrow\infty}\langle Ax_n, Bx_n\rangle = 0.
\end{align*}
Some authors extended the well known result of Bhatia--\v{S}emrl (see \cite{A.R.1,A.R.2,B.G,W,Z.2}).
Very recently, the numerical radius Birkhoff–-James orthogonality in $\mathbb{B}(\mathcal{H})$
has been studied in \cite{M.P.S} as our work was in progress.
In fact, Mal et al. \cite[Theorem 2.3]{M.P.S} obtained the following characterization
of the numerical radius Birkhoff–-James orthogonality for Hilbert space operators:
if $A, B\in \mathbb{B}(\mathcal{H})$, then $A\perp_{B}^{w} B$ if and only if
for each $\theta \in [0, 2\pi)$,
there exists a sequence $\{x_n\}$ of unit
vectors in $\mathcal{H}$ such that
\begin{align*}
\displaystyle{\lim_{n\rightarrow \infty}}\big|\langle Ax_n, x_n\rangle\big| = w(A)
\quad \mbox{and} \quad
\displaystyle{\lim_{n\rightarrow \infty}}
\mbox{Re}\big(e^{i\theta}\langle x_n, Ax_n\rangle\langle Bx_n, x_n\rangle\big)\geq 0.
\end{align*}
In what follows we shall develop the above result for elements of a $C^*$-algebra.
\begin{theorem}\label{T.7.2}
Let $a, b\in \mathfrak{A}$.
Then the following statements are equivalent:
\begin{itemize}
\item[(i)] $a\perp_{B}^{v} b$.
\item[(ii)] For each $\theta \in [0, 2\pi)$, there exists a state $\varphi_{_{\theta}}$ on $\mathfrak{A}$ such that
\begin{align*}
|\varphi_{_{\theta}}(a)| = v(a) \quad \mbox{and} \quad {\rm Re}\big(e^{i\theta}\overline{\varphi_{_{\theta}}(a)}\varphi_{_{\theta}}(b)\big) \geq 0.
\end{align*}
\end{itemize}
\end{theorem}
\begin{proof}
(i)$\Rightarrow$(ii) Let $a\perp_{B}^{v} b$. Hence
$v(a + re^{i\theta}b)\geq v(a)$ for all $\theta \in [0, 2\pi)$ and $r\in \mathbb{R}^{+}$.
Fix $\theta$ and let $b_{\theta} = e^{i\theta}b$. Then we have
$v(a + rb_{\theta})\geq v(a)$ for all $r\in \mathbb{R}^{+}$.
By Theorem \ref{T.3.2} there exists a state $\varphi_{_{\theta}}$ on $\mathfrak{A}$ such that
$|\varphi_{_{\theta}}(a)| = v(a)$ and $\mbox{Re}\big(\overline{\varphi_{_{\theta}}(a)}\varphi_{_{\theta}}(b_{\theta})\big) \geq 0$.
From this it follows that
$|\varphi_{_{\theta}}(a)| = v(a)$ and $\mbox{Re}\big(e^{i\theta}\overline{\varphi_{_{\theta}}(a)}\varphi_{_{\theta}}(b)\big) \geq 0$.

(ii)$\Rightarrow$(i) Suppose (ii) holds.
Let $\lambda \in \mathbb{C}$. Then there exists $\theta \in [0, 2\pi)$ such that
$\lambda = |\lambda|e^{i\theta}$. Therefore, there exists a state $\varphi_{_{\theta}}$ on $\mathfrak{A}$ such that
$|\varphi_{_{\theta}}(a)| = v(a)$ and $\mbox{Re}\big(e^{i\theta}\overline{\varphi_{_{\theta}}(a)}\varphi_{_{\theta}}(b)\big) \geq 0$. Thus
\begin{align}\label{I.1.T.7.2}
v^2(a + \lambda b) &\geq \big|\varphi_{_{\theta}}(a + |\lambda|e^{i\theta} b)\big|^2\nonumber
\\&= |\varphi_{_{\theta}}(a)|^2 + 2|\lambda|\mbox{Re}\big(e^{i\theta}\overline{\varphi_{_{\theta}}(a)}\varphi_{_{\theta}}(b)\big)
+ |\lambda|^2|\varphi_{_{\theta}}(b)|^2\nonumber
\\&\geq v^2(a) + |\lambda|^2|\varphi_{_{\theta}}(b)|^2
\\& \geq v^2(a),\nonumber
\end{align}
and so $v(a + \lambda b)\geq v(a)$. Hence $a\perp_{B}^{v} b$.
\end{proof}
Recall that (e.g., see \cite[p. 63]{G.R}) the {\it Crawford number} of $B\in\mathbb{B}(\mathcal{H})$ is defined by
\begin{align}\label{Craw}
c(B) := \inf \big\{|\langle Bx, x\rangle|:x\in \mathcal{H},\|x\| =1\big\}.
\end{align}
This concept is useful in studying linear operators (see \cite{G.R},
and further references therein).
The {\it numerical radius Crawford number} of $b\in \mathfrak{A}$ can be defined by
\begin{align*}
\mathcal{C}(b) := \inf\big\{|\varphi(b)|: \, \varphi \in \mathcal{S}(\mathfrak{A})\big\}.
\end{align*}
Notice that for $B\in\mathbb{B}(\mathcal{H})$, by \cite[Remark 2.3]{Z.1},
$\mathcal{C}(B)$ coincides with the classical $c(B)$ given by (\ref{Craw}) above.

Before we present the next results, some examples are appropriate.
More precisely, the proposition below gives a large family of elements satisfying $\mathcal{C}(b)>0$.
\begin{proposition}
Let $a\!\in\!\mathfrak{A}$ with $v(a)<1$. If $b:=e+a$, then $\mathcal{C}(b)>0$.
\end{proposition}
\begin{proof}
Since $v(a)<1$, it follows that there is a positive number $\gamma$ such that $v(a)<\gamma<1$.
Fix $\varphi\! \in\! \mathcal{S}(\mathfrak{A})$. Then we obtain
\begin{center}
$|\varphi(b)|=|\varphi(e)+\varphi(a)|\geq |\varphi(e)|-|\varphi(a)|\geq 1-v(a)>1-\gamma$
\end{center}
and passing to the infimum over $\mathcal{S}(\mathfrak{A})$ we obtain
$\mathcal{C}(b)\!\geq\! 1\!-\!\gamma$. So $\mathcal{C}(b)\!>\!0$.
\end{proof}
Now, as an immediate consequence of Theorem \ref{T.7.2}, we have the following result.
\begin{corollary}\label{C.8.2}
Let $a, b\in \mathfrak{A}$.
Then the following statements are equivalent:
\begin{itemize}
\item[(i)] $a\perp_{B}^{v} b$.
\item[(ii)] $v^2(a + \lambda b)\geq v^2(a) + |\lambda|^2 \mathcal{C}^2(b)$ for all $\lambda \in \mathbb{C}$.
\end{itemize}
\end{corollary}
\begin{proof}
If $a\perp_{B}^{v} b$, then for each $\lambda \in \mathbb{C}$, by (\ref{I.1.T.7.2}),
there exists a state $\varphi_{_{\theta}}$ on $\mathfrak{A}$ such that
$v^2(a + \lambda b) \geq v^2(a) + |\lambda|^2|\varphi_{_{\theta}}(b)|^2$. Hence
$v^2(a + \lambda b)\geq v^2(a) + |\lambda|^2 \mathcal{C}^2(b)$.
The converse is obvious.
\end{proof}
The following result is a kind of Pythagorean inequality in $C^*$-algebras.
We are going to apply this inequality in approximation theory.
\begin{proposition}\label{P.9.2}
Let $a, b\in \mathfrak{A}$ with $\mathcal{C}(b)>0$.
Then there exists a unique $\zeta \in \mathbb{C}$, such that
\begin{align*}
v^2\Big((a + \zeta b) + \lambda b\Big)\geq v^2(a + \zeta b) + |\lambda|^2 \mathcal{C}^2(b)
\end{align*}
for all $\lambda \in \mathbb{C}$.
\end{proposition}
\begin{proof}
Since $v(a + \lambda b)$ is large for $|\lambda|$ large,
$\inf\big\{v(a + \lambda b) : \,\lambda \in \mathbb{C}\big\}$ must be attained at some point, say $\zeta$
(there may be of course many such points). Therefore, $v(a + \lambda b) \geq v(a + \zeta b)$ for all
$\lambda \in \mathbb{C}$ and hence $(a + \zeta b)\perp_{B}^{v} b$.
So, by Corollary \ref{C.8.2}, we have
\begin{align*}
v^2\Big((a + \zeta b) + \lambda b\Big)\geq v^2(a + \zeta b) + |\lambda|^2 \mathcal{C}^2(b),
\end{align*}
for all $\lambda \in \mathbb{C}$.
Now, suppose that $\eta$ is another point satisfying the inequality
\begin{align*}
v^2\Big((a + \eta b) + \lambda b\Big)\geq v^2(a + \eta b) + |\lambda|^2 \mathcal{C}^2(b),
\end{align*}
for all  $\lambda \in \mathbb{C}$.
Choose $\lambda = \zeta - \eta$ to get
\begin{align*}
v^2(a + \zeta b) &= v^2\Big((a + \eta b) + (\zeta - \eta)b \Big)
\\& \geq v^2(a + \eta b) + |\zeta - \eta|^2 \mathcal{C}^2(b)
\\& \geq v^2(a + \zeta b) + |\zeta - \eta|^2 \mathcal{C}^2(b).
\end{align*}
Hence $0 \geq |\zeta - \eta|^2 \mathcal{C}^2(b)$. Since $\mathcal{C}(b)> 0$, we get $|\zeta - \eta|^2 = 0$,
or equivalently, $\eta = \zeta$. This shows that $\zeta$ is unique.
\end{proof}
Now we apply the above result to present a theorem concerning
uniqueness of best approximation with respect to the numerical radius norm in $C^*$-algebras.
Similar investigations have been worked out in compact operators spaces for
injective operators (cf. \cite[Theorems 5.6, 5.7, 5.8]{W2}).
\begin{theorem}
Let $b\in \mathfrak{A}$ with $\mathcal{C}(b)>0$.
Then any $a\!\in\!\mathfrak{A}\setminus{\rm span}\{b\}$ has a unique best approximation in
${\rm span}\{b\}$ with respect to the numerical radius norm, that is, there exists a unique $b_a\!\in\!{\rm span}\{b\}$ such that
${\rm dist}(a,{\rm span}\{b\})=v(a-b_a)$.
\end{theorem}
\begin{proof}
Fix $a\!\in\!\mathfrak{A}\setminus{\rm span}\{b\}$. It follows from
Proposition \ref{P.9.2} that there exists a unique $\zeta\!\in\!\mathbb{C}$ such that
$v^2\Big((a + \zeta b) + \lambda b\Big)\geq v^2(a + \zeta b) + |\lambda|^2 \mathcal{C}^2(b)$ for
all $\lambda\!\in\!\mathbb{C}$. If $\lambda\neq 0$, then, by the inequality $\mathcal{C}(b)>0$, we get
the following inequality
\begin{align}\label{best-appr-01}
v\left((a + \zeta b) + \lambda b\right)>v(a + \zeta b)\quad {\rm for\ all}\quad \lambda\!\in\!\mathbb{C}\setminus\{0\}.
\end{align}
Define $b_a:=-\zeta b$. Now the property (\ref{best-appr-01}) becomes
\begin{align*}
v(a - p)>v(a - b_a)\quad {\rm for\ all}\quad p\!\in\!{\rm span}\{b\}\setminus\{b_a\},
\end{align*}
which means that ${\rm dist}(a,{\rm span}\{b\})=v(a-b_a)$.
\end{proof}
In \cite{A.K}, for $A, B\in \mathbb{B}(\mathcal{H})$, a necessary and sufficient condition for
the equality $w(A + B) = w(A) + w(B)$ has been given. In fact, it has been shown
that $w(A + B) = w(A) + w(B)$ if and only if there exists a sequence $\{x_n\}$ of unit
vectors in $\mathcal{H}$ such that
\begin{align*}
\displaystyle{\lim_{n\rightarrow \infty}}
\langle x_n, Ax_n\rangle\langle Bx_n, x_n\rangle = w(A)\,w(B).
\end{align*}
In the following theorem, we give a necessary and sufficient condition
for the equality $v(a + b) = v(a) + v(b)$ in $C^*$-algebras.
\begin{theorem}\label{T.10.2}
Let $a, b \in\mathfrak{A}$.
Then the following statements are equivalent:
\begin{itemize}
\item[(i)] $v(a + b) = v(a) + v(b)$.
\item[(ii)] There exists a state $\varphi$ on $\mathfrak{A}$ such that
$\overline{\varphi(a)} \varphi(b) = v(a)v(b)$.
\end{itemize}
\end{theorem}
\begin{proof}
(i)$\Rightarrow$(ii) Let $v(a + b) = v(a) + v(b)$.
By Proposition 4.1 of \cite{A.R.2}, we get $a\perp_{B}^{v} \big(v(a)b - v(b)a\big)$.
Therefore, by Theorem \ref{T.3.2},
there exists a state $\varphi$ on $\mathfrak{A}$ such that
$|\varphi(a)| = v(a)$ and $\mbox{Re}\Big(\overline{\varphi(a)}\varphi\big(v(a)b - v(b)a\big)\Big) \geq 0$.
This implies $v(a)v(b) \leq \mbox{Re}\big(\overline{\varphi(a)}\varphi(b)\big)$.
Consequently,
\begin{align*}
v(a)v(b) \leq \mbox{Re}\big(\overline{\varphi(a)}\varphi(b)\big)
\leq \big|\overline{\varphi(a)}\varphi(b)\big| \leq v(a)v(b),
\end{align*}
which yields $\mbox{Re}\big(\overline{\varphi(a)}\varphi(b)\big) = v(a)v(b)$
and $\mbox{Im}\big(\overline{\varphi(a)}\varphi(b)\big) = 0$.
Hence $\overline{\varphi(a)} \varphi(b) = v(a)v(b)$.

(ii)$\Rightarrow$(i) Suppose (ii) holds. So, there exists a state $\varphi$ on $\mathfrak{A}$ such that
$\overline{\varphi(a)} \varphi(b) = v(a)v(b)$. From this it follows that
$|\varphi(a)| = v(a)$ and $|\varphi(b)| = v(b)$. Therefore, we have
\begin{align*}
\big(v(a) + v(b)\big)^2 &= |\varphi(a)|^2 + 2\overline{\varphi(a)} \varphi(b) + |\varphi(b)|^2
\\&= |\varphi(a + b)|^2 \leq v^2(a + b) \leq \big(v(a) + v(b)\big)^2,
\end{align*}
and so $v(a + b) = v(a) + v(b)$.
\end{proof}
Applying the above result we may prove another theorem.
\begin{theorem}
Let $a, b,c \in\mathfrak{A}\setminus\{0\}$.
Then the following statements are equivalent:
\begin{itemize}
\item[(i)] $v(a + b+c) = v(a) + v(b)+v(c)$.
\item[(ii)] There exists a state $\varphi$ on $\mathfrak{A}$ such
that $\frac{\varphi(a)}{v(a)}\!=\!\frac{\varphi(b)}{v(b)}\!=\!\frac{\varphi(c)}{v(c)}$ and\\
$\overline{\varphi(a)} \varphi(b)\! =\! v(a)v(b)$,\quad
$\overline{\varphi(a)} \varphi(c)\! =\! v(a)v(c)$,\quad
$\overline{\varphi(b)} \varphi(c)\! =\! v(b)v(c)$.
\item[(iii)] There exists a state $\psi$ on $\mathfrak{A}$ such
that $\frac{\psi(a)}{v(a)}\!=\!\frac{\psi(b)}{v(b)}\!=\!\frac{\psi(c)}{v(c)}$ and
$\left|\frac{\psi(a)}{v(a)}\right|\!=\!1$.
\end{itemize}
\end{theorem}
\begin{proof}
(i)$\Rightarrow$(ii) It is known that the norm
equality $v(a + b+c) = v(a) + v(b)+v(c)$ holds if and only
if $v(\alpha a + \beta b+\gamma c) = v(\alpha a) + v(\beta b)+v(\gamma c)$ for all $\alpha,\beta,\gamma\geq 0$.
We assume (i), so without loss of generality, we may assume that $v(a)=v(b)=v(c)=1$ and $v(a + b+c)=3$.
Since
\begin{center}
$3=v(a + b+c)\leq v(a) + v(b+c)\leq v(a) + v(b)+v(c)=3$,
\end{center}
we have $v(a+(b+c))=v(a)+v(b+c)$ and $v(b+c)=2$.
By Theorem \ref{T.10.2}, there is a state $\varphi$ on $\mathfrak{A}$ such
that $\overline{\varphi(a)}\varphi(b+c)=v(a)v(b+c)=2$. It follows
that $\frac{1}{2}\overline{\varphi(a)}\varphi(b)+\frac{1}{2}\overline{\varphi(a)}\varphi(c)=1$.

We know that those three numbers
$\overline{\varphi(a)}\varphi(b)$, $\overline{\varphi(a)}\varphi(c)$, $1$ are
in $\{\xi\!\in\!\mathbb{C}:|\xi|\!=\!1\}$. Since one of them is a
convex combination of the others, they must all be the same scalar. Therefore
$\overline{\varphi(a)}\varphi(b)=\overline{\varphi(a)}\varphi(c)=1$.
It follows easily that $\varphi(a)\overline{\varphi(b)}=1$. Multiplying these equalities we have
$\varphi(a)\overline{\varphi(b)}\overline{\varphi(a)}\varphi(c)=1$. Since
$\varphi(a)\overline{\varphi(a)}=1$, we get $\overline{\varphi(b)} \varphi(c)\! =\! 1$.
To summarize, it has been shown that $\overline{\varphi(a)} \varphi(b)\! =\!1$,
$\overline{\varphi(a)} \varphi(c)\! =\! 1$, $\overline{\varphi(b)} \varphi(c)\! =\! 1$.
If we divide both sides of the equality $\overline{\varphi(a)}\varphi(b)=\overline{\varphi(a)}\varphi(c)$
by $\overline{\varphi(a)}$, we obtain $\varphi(b)=\varphi(c)$. Similarly,
since $\varphi(a)\overline{\varphi(c)}=\varphi(b)\overline{\varphi(c)}$, we
get $\varphi(a)=\varphi(b)$. The proof of the implications (i)$\Rightarrow$(ii) is complete.

The implication (ii)$\Rightarrow$(iii) is trivial. So we prove (iii)$\Rightarrow$(i).
Assume that (iii) holds. Again, we may assume that $v(a)\!=\!v(b)\!=\!v(c)\!=\!1$.
It follows from (iii) that $|\psi(a)|\!=\!|\psi(b)|\!=\!|\psi(c)|=1$. Further, from the
condition (iii) we have
\begin{align*}
3&=3|\psi(a)|=|\psi(a)+\psi(a)+\psi(a)|=|\psi(a)+\psi(b)+\psi(c)|\\
&=|\psi(a+b+c)|\leq v(a+b+c)\leq v(a)+v(b)+v(c)=3.
\end{align*}
So, the inequalities become equalities and the proof is complete.
\end{proof}
It is worth mentioning that investigations with more than two elements
have been appeared in \cite{W2}, but for sum of operators.
\begin{remark}\label{R.11.2}
In \cite[Theorem 2.2]{Z.1}, the following
characterization of the numerical radius for elements of a $C^*$-algebra has been given,
\begin{align*}
v(a) = \displaystyle{\sup_{\theta \in \mathbb{R}}}\|\mbox{Re}(e^{i\theta}a)\|.
\end{align*}
Then, a refinement of the triangle inequality
for the numerical radius in $C^*$-algebras has been shown in \cite[Theorem 3.6]{Z.1}
that for every $a, b \in\mathfrak{A}$,
\begin{align}\label{I.1.R.11.2}
v(a + b) &\leq \frac{1}{2} \big(v(a) + v(b)\big)\nonumber
\\& \qquad + \frac{1}{2}\sqrt{\big(v(a) - v(b)\big)^2 +
4 \sup_{\theta \in \mathbb{R}}\big\|\mbox{Re}(e^{i\theta}a)\mbox{Re}(e^{i\theta} b)\big\|}
\\& \leq v(a) + v(b)\nonumber.
\end{align}
Therefore, Theorem \ref{T.10.2} implies that, for every $a, b \in\mathfrak{A}$,
if there exists a state $\varphi$ on $\mathfrak{A}$ such that
$\overline{\varphi(a)} \varphi(b) = v(a)v(b)$, then $v(a + b) = v(a) + v(b)$
and consequently by (\ref{I.1.R.11.2}) we get
\begin{align*}
\displaystyle{\sup_{\theta \in \mathbb{R}}}\big\|\mbox{Re}(e^{i\theta}a)\mbox{Re}(e^{i\theta}b)\big\|
= v(a)v(b) = \overline{\varphi(a)} \varphi(b).
\end{align*}
\end{remark}
\textbf{Acknowledgement.}
The authors would like to thank the referees for their valuable suggestions and comments.
The research of Pawe\l{} W\'ojcik and this paper were
partially supported by National Science
Centre, Poland under Grant Miniatura~2, No. 2018/02/X/ST1/00313.
\bibliographystyle{amsplain}

\end{document}